\documentclass[12pt]{amsart}
\usepackage[T1]{fontenc}
\usepackage{amsmath, amssymb}
\usepackage[english]{babel}
\usepackage{mathrsfs}
\usepackage[all]{xy}
\usepackage{enumitem}
\usepackage{url}

\usepackage{tikz}

\usepackage[bookmarks=true, colorlinks=true, linkcolor=blue, citecolor= red,hypertexnames=false,hyperindex,breaklinks]{hyperref}
\usepackage[capitalize]{cleveref}

\newtheorem{thm}{Theorem}[section]
\newtheorem*{thm*}{Theorem}
\newtheorem{lem}[thm]{Lemma}
\newtheorem{fact}[thm]{Fact}
\newtheorem{facts}[thm]{Facts}
\newtheorem{prop}[thm]{Proposition}
\newtheorem*{prop*}{Proposition}

\newtheorem*{stat*}{Statement}

\newtheorem{cor}[thm]{Corollary}

\theoremstyle{definition}
\newtheorem{defn}[thm]{Definition}

\newtheorem{remark}[thm]{Remark}

\newtheorem{quest}[thm]{Question}

\def\e{\varepsilon}

\newcommand\ip[2]{\left\langle #1\, ,\!\ #2 \right\rangle}

\def \bbR{\mathbb{R}}
\def \Rb{\mathbb{R}}
\def \bbQ{\mathbb{Q}}

\def \Nb{\mathbb{N}}

\crefname{main}{Theorem}{Theorems}
\crefname{mconj}{Conjecture}{Conjectures}
\crefname{mcor}{Corollary}{Corollaries}
\crefname{thm}{Theorem}{Theorems}
\crefname{thm*}{Theorem}{Theorems}
\crefname{lem}{Lemma}{Lemmas}
\crefname{fact}{Fact}{Facts}
\crefname{facts}{Facts}{Factss}
\crefname{prop}{Proposition}{Propositions}
\crefname{prop*}{Proposition}{Propositions}
\crefname{stat}{Statement}{Statements}
\crefname{stat*}{Statement}{Statements}
\crefname{conj}{Conjecture}{Conjectures}
\crefname{cor}{Corollary}{Corollaries}
\crefname{princ}{Principle}{Principles}
\crefname{defn}{Definition}{Definitions}
\crefname{notation}{Notation}{Notations}
\crefname{remark}{Remark}{Remarks}
\crefname{remarks}{Remarks}{Remarks}
\crefname{question}{Question}{Questions}
\crefname{quest}{Question}{Questions}
\crefname{example}{Example}{Examples}
\crefname{claim}{Claim}{Claims}

\providecommand{\monus}{
  \mathbin{
    \vphantom{+}
    \text{
      \mathsurround=0pt
      \ooalign{
        \noalign{\kern-.35ex}
        \hidewidth$\smash{\cdot}$\hidewidth\cr
        \noalign{\kern.35ex}
        $-$\cr
      }
    }
  }
}

\newcommand{\ACN}{\mathrm{AC}_{\Nb}}
\newcommand{\ACNN}{\mathrm{AC}_{\Nb,\Nb}}
\newcommand{\ACNT}{\mathrm{AC}_{\Nb,2}}

\begin{document}


\title{Pointwise Mean Value Theorems in Constructive Mathematics}

\author{Jananan Arulseelan and James E. Hanson}
\address{Department of Mathematics, Iowa State University, 396 Carver Hall, 411 Morrill Road, Ames, IA 50011, USA}
\email{jananan@iastate.edu}
\urladdr{https://sites.google.com/view/jananan-arulseelan}

\address{Department of Mathematics, Iowa State University, 396 Carver Hall, 411 Morrill Road, Ames, IA 50011, USA}
\email{jameseh@iastate.edu}
\urladdr{https://james-hanson.github.io/}

\begin{abstract}

  We answer some questions regarding the mean value theorem and related results in constructive mathematics.  The answers to these questions reveal interesting properties of the Mean Value Theorem, Law of Bounded Change, and Constancy Principle.  We see that, in contrast to the Intermediate Value Theorem whose approximate analogue was shown to hold constructively by Frank, the natural approximate versions of these theorems fail to hold in neutral constructive mathematics.  Our proof of this makes use of the existence of a topos in which the real numbers are in bijection with the naturals, which was shown by Bauer and the second author.  Using Booij's notion of locators, we show that the aforementioned approximate analogues do hold in the presence of a small amount of countable choice, and also under suitable locator lifting hypotheses. We also show that an even weaker approximate analogue of the Mean Value Theorem holds in neutral constructive mathematics.



\end{abstract}

\maketitle


\section{Introduction}\label{SectionIntro}

We consider the following statement, referred to henceforth as the Approximate Mean Value Theorem (or AMVT). Note that here and throughout the paper, continuity and differentiability are understood in the sense of pointwise continuity and pointwise differentiability. 

\begin{thm*}[AMVT]
  For any $a < b$ and $f: [a,b] \to \mathbb{R}$, if $f$ is continuous on $[a,b]$ and differentiable on $(a,b)$, then for any $\e > 0$, there exists $c \in (a,b)$ such that 
  \[
    \left| f'(c) - \frac{f(b)-f(a)}{b-a} \right| < \e.
  \]
\end{thm*}

As can be shown by the average student in a first course in real analysis, the Approximate Mean Value Theorem is true under the assumptions of classical mathematics. 

We  investigate the extent to which the Approximate Mean Value Theorem is true in more constructive settings.  In particular, we show in \cref{countex} that it is consistent with neutral constructive mathematics that the AMVT fails.

By contrast, we show in \cref{sec:WAMVT} that the following even weaker analogue of the Mean Value Theorem (WAMVT, \cref{WAMVT}) is provable in neutral constructive mathematics. 

\begin{thm*}[WAMVT]
  For any $a < b$ and $f: [a,b] \to \mathbb{R}$, if $f$ is continuous on $[a,b]$, then for any $\e > 0$ and $\delta > 0$, there exists $c, d \in (a,b)$ such that $c < d$, $d - c < \delta$, and
  \[
    \left| \frac{f(d) - f(c)}{d - c} - \frac{f(b)-f(a)}{b-a} \right| < \e.
  \]
\end{thm*}

This and other results in our paper depend heavily on the constructive version of the intermediate value theorem for pointwise continuous functions proven by Frank in \cite{Frank}.

Inspired by the proof of Hardin and Velleman \cite{HV} that the Mean Value Theorem is provable in RCA$_0$, we show in \cref{AMVT} that the conclusion of the AMVT is true for functions that lift to locators in the sense of Booij.  This allows us to identify several weak forms of the axiom of countable choice that are each sufficient to prove the AMVT.

We use these results to also study the Law of Bounded Change,

\begin{thm*}
  For any $a < b$ and $f: [a,b] \to \mathbb{R}$, if $f$ is continuous on $[a,b]$ and differentiable on $(a,b)$ and if $|f'(x)| \leq M$ for all $x \in (a,b)$, then $|f(b) - f(a)| \leq M|b - a|$.
\end{thm*}

and the constancy principle, stated below.

\begin{thm*}
  Let $f: \bbR \to \bbR$ be a pointwise differentiable function with $f'(x) = 0$ for all $x$.  Then $f$ is constant.
\end{thm*}

We should note that constructive versions of the mean value theorem for uniformly continuous and uniformly differentiable functions were considered by Bishop and Bridges in \cite{Bishop2011-nx} and by Coquand and Spitters in \cite{Coquand2012}.

\section{Counterexample in Neutral Constructive Mathematics}
\label{countex}


Neutral constructive mathematics refers to mathematics done without the law of excluded middle (LEM) and without the axiom of choice. It is therefore ``neutral'' in that it is agnostic regarding the truth of LEM and AC. It assumes no extra axioms or weakenings of the ones that have been excised.  

We show in this section that there exist explicit countermodels to the AMVT and the Constancy Principle in neutral constructive mathematics.  Namely, we show that, in any model admitting a surjection from $\mathbb{N}$ to the Dedekind reals, the AMVT fails.  Such models were shown to exist by Bauer and the second author in \cite{BH}.  

We make use of the following lemma.  Intuitively it asserts, given a finite list $[a_0, b_0], \ldots, [a_{n-1}, b_{n-1}]$ of subintervals of $[0,1]$, the existence of a unique function which behaves suitably like the integral $\int^{x}_{0} \chi_{[a_0, b_0] \cup \ldots \cup [a_{n-1}, b_{n-1}]}(y) \operatorname{d}y$.

\begin{lem}\label[lem]{gConstruction}
  For any finite list $[a_0, b_0], \ldots, [a_{n-1}, b_{n-1}]$ of subintervals of $[0,1]$, there exists a function $g: [0,1] \to [0,1]$ satisfying:
  \begin{enumerate}
    \item $g(0) = 0$,
    \item $g$ is non-decreasing and $1$-Lipschitz,
    \item for all $i < n$ and $x$, $y$ with $a_i \leq x \leq y \leq b_i$, we have $g(y) - g(x) = y - x$, and
    \item $g(x)$ is constant on any open interval disjoint from $[a_0, b_0] \cup \ldots \cup [a_{n-1}, b_{n-1}]$.
  \end{enumerate}
  Moreover, the above conditions uniquely specify $g$.
\end{lem}

\begin{proof}
  See Appendix~\ref{sec:proof-of-lemma}.
\end{proof}


The next lemma is immediate from the construction in Lemma~\ref{gConstruction}.

\begin{lem}\label[lem]{gBounds}
  Define $g_{[a_0, b_0], \ldots, [a_{n-1}, b_{n-1}]}$ as in the above lemma. For any intervals $[a_0, b_0], \ldots, [a_{n}, b_{n}]$ and $x \in [0,1]$, we have 
  \[
    | g_{[a_0, b_0], \ldots, [a_{n-1}, b_{n-1}]}(x) - g_{[a_0, b_0], \ldots, [a_{n}, b_{n}]}(x) | \leq |b_n - a_n |.
  \]
\end{lem}


We are now equipped to construct a counterexample to the Constancy Principle.

\begin{prop}
  If there is a surjection from $\mathbb{N}$ onto $\mathbb{R}$, then for any $\e > 0$, there is a non-decreasing $1$-Lipschitz function $f:[0,1] \to [0,1]$ such that $f(0)=0$, $f(1) > 1 - \e$, and for every $x \in [0,1]$, $f$ is locally constant at $x$.
\end{prop}

\begin{proof}
  The existence of a surjection from $\mathbb{N}$ onto $\mathbb{R}$ yields a surjection from $\mathbb{N}$ onto $[0,1]$. Fix an enumeration $(c_n)_{n \in \mathbb{N}}$ of $[0,1]$. Fix a sequence $(\e_n)_{n \in \mathbb{N}}$ of positive real numbers such that $3 \sum^{\infty}_{n=0} \e_n < \e$. For each $n$, set $a_n = \max\{c_n - \frac{1}{2}\e_n, 0\}$ and $b_n = \min\{c_n + \frac{1}{2}\e_n, 1\}$. By construction, $a_n \leq b_n$, $|b_n - a_n| \leq \e_n$, and $[a_n, b_n]$ is a subinterval of $[0,1]$.

  For each $n \in \mathbb{N}$, define $f_n(x) = x - g_{[a_0, b_0], \ldots ,[a_{n-1}, b_{n-1}]}(x)$ as in Lemma~\ref{gConstruction}. Since $g$ is uniquely specified by Lemma~\ref{gConstruction}, we can do this without invoking any choice principles. By Lemma~\ref{gBounds}, for each $n$ and for each $x \in [0,1]$, we have $|f_n(x) - f_{n-1}(x)| \leq 2\e_{n-1}$, so the sequence $(f_n)_{n \in \mathbb{N}}$ converges uniformly. Consider the limit $f:[0,1] \to [0,1]$ of this sequence. We see that $f$ is non-decreasing and $1$-Lipschitz, and since $f_n(0)=0$ for each $n$, we have that $f(0) = 0$. Since $f_0(1) = 1$, we have by construction that 
  \[
    f(1) \geq 1 - \frac{2}{3}\e > 1-\e.
  \]

  It remains to show that $f(x)$ is locally constant at each Dedekind real. To this end, fix $r \in [0,1]$ and find an $m$ such that $c_m = r$. By construction, for all $n > m$, we have that $f_n(x)$ is constant on the interval $(a_m, b_m)$ containing $r$. Therefore $f(x)$ is also constant on the interval $(a_m, b_m)$.
\end{proof}

By \cite{BH}, it is consistent with neutral constructive mathematics that there is a surjection from $\mathbb{N}$ to $\mathbb{R}$. Therefore, the argument above shows that it is consistent with neutral constructive mathematics that there is a pointwise differentiable function $f$ satisfying $f'(x) = 0$ for all $x$ but $f$ is not constant.

\begin{thm}
  The AMVT is not provable in neutral constructive mathematics.
\end{thm}

\begin{proof}
  By the reasoning above, it is consistent with neutral constructive mathematics that there is a locally constant function $f: [0,1] \to [0,1]$ with $f(0) = 0$ and $f(1) > 0$.  We may assume $f(1) = 1$ by scaling.  For such $f$, picking $\e > 0$ such that $\e < 1$, we see that for all $c \in (0,1)$ we have 
  \[
    \left| f'(c) - \frac{f(1) - f(0)}{1-0}\right| = |0 - 1| = 1 > \e.\qedhere 
  \]
\end{proof}

In constructive analysis, it is not uncommon for unprovable statements that are classically true to become provable after changing pointwise conditions, such as pointwise continuity, to some uniform analog. Our counterexample above is $1$-Lipschitz and monotone, although it is not absolutely continuous (since it is essentially Cantor's function modified so that the function is `only increasing at points that do not exist'). 

\begin{quest}
  Does the Constancy Principle hold for absolutely continuous functions in neutral constructive mathematics?
\end{quest}



\section{Weak Approximate Mean Value Theorem}
\label{sec:WAMVT}

We can use finitely many applications of locatedness (i.e., the property that for any rational $r < s$, either $r < x$ or $x < s$) with the same ideas as the proof of \ref{approxRolle} to get a weak approximate Rolle's theorem and thereby the weak approximate mean value theorem. 

\begin{figure}
  \begin{tikzpicture}
    \draw[thick] (-5,2) .. controls (-2,6.5) and (-0.5,5.2) .. (0,4.8) .. controls (1,4.2) and (3,2.2) .. (5,1);
    \draw (-6,0) -- (6,0);

    \draw[dashed, gray] (-4,3.4) -- (1.8,3.4);
    \draw[dashed, gray] (-4.5,2.7) -- (2.7,2.7);

    \filldraw (-5,0) circle (2pt);
    \filldraw (0,0) circle (2pt);
    \filldraw (5,0) circle (2pt);
    \filldraw (-4,0) circle (2pt);
    \filldraw (1.8,0) circle (2pt);
    \filldraw (-4.5,0) circle (2pt);
    \filldraw (2.7,0) circle (2pt);

    \filldraw (-5,2) circle (2pt) node[left] {$f(e_n)$};
    \filldraw (0,4.8) circle (2pt) node[above right] {$f(e_{n+1})$};
    \filldraw (5,1) circle (2pt) node[right] {$f(e_{n+2})$};
    \filldraw (-4,3.4) circle (2pt) node[above left] {$f(c')$};
    \filldraw (1.8,3.4) circle (2pt) node[above right] {$f(d')$};
    \filldraw (-4.5,2.7) circle (2pt) node[left, yshift=4pt] {$f(c)$};
    \filldraw (2.7,2.7) circle (2pt) node[right, yshift=4pt] {$f(d)$};

    \node[anchor=base] at (-5,-0.5) {$e_n$};
    \node[anchor=base] at (0,-0.5) {$e_{n+1}$};
    \node[anchor=base] at (5,-0.5) {$e_{n+2}$};
    \node[anchor=base] at (-4,-0.5) {$c'$};
    \node[anchor=base] at (1.8,-0.5) {$d'$};
    \node[anchor=base] at (-4.5,-0.5) {$c$};
    \node[anchor=base] at (2.7,-0.5) {$d$};
  \end{tikzpicture}
  \caption{The construction from the proof of \cref{weakApproxRolle}. $c'$ and $d'$ are found first using the constructive approximate IVT. Then $c$ and $d$ are found in the same way. Choosing $c'$ and $d'$ first ensures a minimum distance of $d' - c'$ between $c$ and $d$.}
  \label{fig:IVT}
\end{figure}
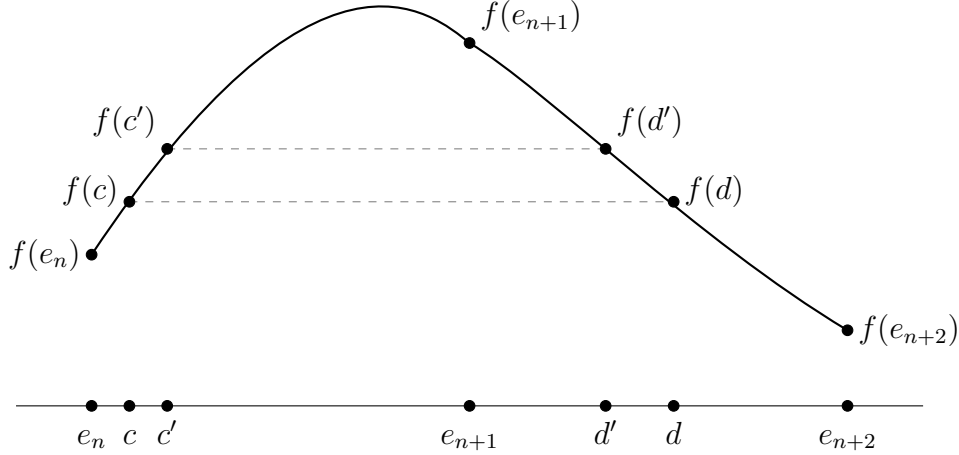

\begin{prop}[Weak Approximate Rolle's Theorem]\label[prop]{weakApproxRolle}
  Fix $a < b$ and a pointwise continuous function $f: [a,b] \to \mathbb{R}$ which satisfies $f(a) = f(b)$. For any $\e,\delta > 0$, we may find rational $c ,d$ with $a < c < d < b$ such that  $|d -c| < \delta$ and $\left|\frac{f(d)-f(c)}{d-c}\right| < \e$.
  %
\end{prop}
\begin{proof}
  Let $e_0 = a$ and $e_N = b$. Fix rational $e_1 < e_2 < \dots < e_{N-1}$ in $(a,b)$ satisfying that for each $n < N-1$, $|e_{n+1} - e_n| = |e_{n+2} - e_{n+1}| < \frac{1}{2}\delta$. Let $r$ be the fixed value of $e_{n+1} - e_n$ for $n \notin \{0,N-1\}$.   

  For each $n < N$, let $\gamma_n = \left|\frac{f(e_{n+1})-f(e_n)}{e_{n+1}-e_{n}}\right|$. For each $n < N$, we have that either $\gamma_n < \frac{2}{3}\e$ or $\gamma_n > \frac{1}{3}\e$. If $\gamma_n < \frac{2}{3}\e$ for any $n$, then we have the required pair $c,d$ with $c = e_n$ (or a rational number sufficiently close to $a$ if $n = 0$) and $d = e_{n+1}$ (or a rational number sufficiently close to $b$ if $n = N-1$). 
  
  Now assume that for all $n < N$, $\gamma_n > \frac{1}{3}\e$. Note that in particular this implies that for each $n < N$, exclusively either $f(e_n) + \frac{1}{3}\e r < f(e_{n+1})$ or $f(e_n) - \frac{1}{3} \e r > f(e_{n+1})$. In the first case, say that $f$ is `going up at $e_n$' and in the second case say that $f$ is `going down at $e_n$'.

  It is straightforward to show that there must be an $n < N$ at which $f$ is going up and an $n < N$ at which $f$ is going down. This implies that one of the following must happen:
  \begin{itemize}
    \item There is an $n < N-1$ such that $f$ is going up at $e_n$ and going down at $e_{n+1}$.
    \item There is an $n < N-1$ such that $f$ is going down at $e_n$ and going up at $e_{n+1}$.
  \end{itemize}
  By replacing $f$ with $-f$ if necessary, we may assume without loss of generality that the first happens. Fix some such $n$. We may also assume without loss of generality that $f(e_n) \geq f(e_{n+2})$.  
  
  By the constructive approximate intermediate value theorem, we can find
  \begin{itemize}
    \item $c' \in (e_n,e_{n+1})$ such that $\left|f(c') - \frac{1}{2}(f(e_n) + f(e_{n+1})) \right| < \frac{1}{6}\e r$ and
  \item $d' \in (e_{n+1},e_{n+2})$ such that $\left|f(d') - \frac{1}{2}(f(e_n) + f(e_{n+1})\right| < \frac{1}{6}\e r$. 
  \end{itemize}
  Clearly $c' < d'$. Moreover, note that $f(e_{n+2}) \leq f(e_{n}) < f(c'),f(d') < f(e_{n+1})$. (See Figure~\ref{fig:IVT}.) 

  Now apply the constructive approximate intermediate value theorem again to find
  \begin{itemize}
    \item $c \in (e_n,c')$ such that $\left|f(c) - \frac{1}{2}(f(e_n) + \min(f(c'),f(d')))\right| < \frac{1}{2}\e |d'-c'|$ and 
    \item $d \in (d',e_{n+2})$ such that $\left|f(d) - \frac{1}{2}(f(e_n) + \min(f(c'),f(d')))\right| < \frac{1}{2}\e |d'-c'|$.
  \end{itemize}
  Since $f$ is pointwise continuous, we may assume that $c$ and $d$ are rational.  By construction $|f(d) - f(c)| < \e |d'-c'| < \e |d - c|$ and $|d-c| < \frac{1}{2}\delta + \frac{1}{2}\delta = \delta$, so we're done. 
\end{proof}

This immediately gives us the weak approximate mean value theorem by the standard argument.

\begin{cor}[Weak Approximate Mean Value Theorem] \label[cor]{WAMVT}
  For any $a < b$ and $f: [a,b] \to \mathbb{R}$, if $f$ is continuous on $[a,b]$, then for any $\e,\delta > 0$, there exists rational $c, d \in (a,b)$ such that $c < d$ and $d - c < \delta$ and
  \[
    \left| \frac{f(d) - f(c)}{d - c} - \frac{f(b)-f(a)}{b-a} \right| < \e.  \]
\end{cor}


\section{Weak Countable Choice Axioms Imply AMVT}
\label{AMVT}

Recall that the \textbf{axiom of countable choice}, or $\ACN$, says that for any family $(X_n)_{n \in \Nb}$ of inhabited sets, there is a function $f$ such that $f(n) \in X_n$ for each $n \in \Nb$. 

There are several weak forms of the axiom of countable choice. The following two are relevant to the current paper. 

\begin{itemize}
  \item $\ACNN$ is the axiom of choice for $\Nb$-indexed families of inhabited subcountable sets (i.e., for every family $(X_n)_{n \in \Nb}$ of subsets of $\Nb$, there is a function $f$ such that $f(n) \in X_n$ for each $n \in \Nb$).
  \item $\ACNT$ is the axiom of choice for $\Nb$-indexed families of inhabited subsets of $2 = \{0,1\}$.  
\end{itemize}

Clearly $\ACN \to \ACNN \to \ACNT$, and moreover $\mathrm{LEM} \to \ACNT$. 

%
%


We make use of the following concept due to Booij \cite{Booij}. 

\begin{defn}
  A \textbf{locator} for a real number $x$ is a function
  \[
    \ell : \{\ip{p}{q} \in \mathbb{Q}^2 : p < q\} \to \{0,1\}
  \]
  such that for any $p,q \in \bbQ$ with $p < q$,
  \begin{itemize}
    \item if $\ell(p,q) = 0$, then $p < x$ and
    \item if $\ell(p,q) = 1$, then $x < q$.
  \end{itemize}
  When $\ell(p,q) = 0$ we say that the locator \textbf{returns $p < x$} and when $\ell(p,q) = 1$ we say that the locator \textbf{returns $x < q$}.

  We say that $x$ \textbf{has} or \textbf{admits a locator} if there exists a locator for $x$.
\end{defn}


Note that in case $q < x < r$, a locator for $x$ can return $q < x$ or return $x < r$ equally well.  Therefore locators are not unique. The real numbers that have locators are precisely the Cauchy reals, so classically, every real has a locator, but it is also constructively consistent that there is a Dedekind real that does not admit a locator. The existence of a map choosing a locator for each real implies the weak limited principle of omniscience (see \cite[Lemma 3.10.1]{Booij} and the paragraph that follows). 





\begin{fact}\label[fact]{fact:real-loc}
  $\ACNT$ implies that every Dedekind real has a locator.
\end{fact}

A Dedekind real $x$ admits a locator if it is a regular Cauchy real (also called a rapidly converging Cauchy real), so Fact~\ref{fact:real-loc} restates the fact that $\ACNT$ implies that the Dedekind and (regular) Cauchy reals agree. 

\begin{defn}\label[defn]{defn:lifts-to-loc}
  A (possibly partial) function $f : {\subseteq}\Rb^n \to \Rb$ is said to \textbf{lift to locators} if there is a function $g$ such that for any reals $x_1,\dots,x_n$ and any $p_1,\dots,p_n$ with $p_i$ a locator of $x_i$, $g(\vec{x},\vec{p})$ is a locator of $f(\vec{x})$. 
\end{defn}

We record the following facts about locators, and use them freely in the sequel.

\begin{facts}\label[facts]{locatorFacts}
  \
  \begin{enumerate}
    \item There is a canonical assignment of locators for rational numbers. 
    \item\label{basic-ops} The functions $-x$, $x+y$, $x\cdot y$, $\frac{1}{x}$, $\max(x,y)$, and $\min(x,y)$ lift to locators. 
  \end{enumerate}
\end{facts}

We fix, once and for all, algorithms for constructing locators for reals that result from basic algebraic operations as in the second bullet above.

Note that lifting to locators is neither stronger nor weaker than continuity. Note also that by Fact \ref{locatorFacts} item~(\ref{basic-ops}), all rational functions with real coefficients lift to locators.


It is easy to see that if every continuous function lifts to locators, then every Dedekind real has a locator.

%

In the following two definitions, we implicitly fix a function $f: [a,b] \to \bbR$ which lifts to locators. We also implicitly fix a locator function $g$ (in the sense of Definition~\ref{defn:lifts-to-loc}) for the function $\left|\frac{f(x) - f(y)}{x-y} \right|$. Given a pair of rationals $x,y$, we will refer to the locator produced by $g$ for $\left|\frac{f(x) - f(y)}{x-y} \right|$ (using the canonical locators for $x$ and $y$) as `the' locator of $\left|\frac{f(x) - f(y)}{x-y} \right|$. 

\begin{defn}
  We say a pair $\langle x,y \rangle$ of rationals in $(a,b)$ with $x < y$ is \textbf{$n$-flat} if $\left| \frac{f(y) - f(x)}{y - x} \right| < 1 - 2^{-2n-1}$.
\end{defn}

Note that being $n$-flat is more restrictive than being $(n+1)$-flat.

\begin{defn}
  Say the pair $\langle x,y \rangle$ of rationals in $(a,b)$ with $x < y$ is \textbf{locator-$n$-flat} if the locator for $\left| \frac{f(y) - f(x)}{y - x} \right|$ returns $\left| \frac{f(y) - f(x)}{y - x} \right| < 1 - 2^{-2n-1}$ when presented the pair $\langle  1 - 2^{-2n}, 1 - 2^{-2n-1} \rangle$.
\end{defn}

Clearly locator-$n$-flatness implies $n$-flatness. Moreover, it is easy to show that $n$-flatness implies locator-$(n+1)$-flatness. The point of locator-$n$-flatness is that it is a decidable property.  


\begin{prop}[Approximate Rolle's Theorem]\label[prop]{approxRolle}
  Fix $a < b$ and a pointwise continuous function $f: [a,b] \to \mathbb{R}$ which lifts to locators and satisfies $f(a) = f(b)$, then for any $\e > 0$ there are sequences of rationals $(a_n)_{n \in \mathbb{N}}$ and $(b_n)_{n \in \mathbb{N}}$ in $(a,b)$ such that for any $n$,
  \begin{itemize}
    \item $a_n \leq a_{n+1} < b_{n+1} \leq b_n$,
    \item $|b_n - a_n| < (\frac{1}{2})^n$, and
    \item $\left| \frac{f(b_n) - f(a_n)}{b_n - a_n} \right| < \e$.
  \end{itemize}
\end{prop}

\begin{proof}
  We may assume without loss of generality that $\e$ is rational, and in fact, by scaling, that $\e = 1$.  We will construct the sequence of pair $\langle a_n, b_n \rangle$ directly, ensuring as we do so that $\langle a_n, b_n \rangle$ is $n$-flat.

  Since $f$ is continuous, $a<b$, and $f(a)=f(b)$, we can find rational $a_0$ and $b_0$ such that $a < a_0 < b_0 < b$ and the pair $\langle a_0, b_0 \rangle$ is $0$-flat.

  Suppose that we have found rational $a_n,b_n \in (a,b)$ which are an $n$-flat pair. By \cref{WAMVT}, there exists rational $c,d \in (a_n,b_n)$ with $c < d$ and $|d-c| < \left(\frac{1}{2}\right)^{n+1}$ such that
\[
  \left| \frac{f(d) - f(c)}{d - c} - \frac{f(b_n) - f(a_n)}{b_n - a_n} \right | < 1 - 2^{-2n-1} - \left| \frac{f(b_n) - f(a_n)}{b_n - a_n}\right|. 
\]
In particular, this implies that $\left| \frac{f(d) - f(c)}{d -c} \right| < 1 - 2^{-2n-1}$. So the pair $\langle c,d\rangle$ is $n$-flat and thereby locator-$(n+1)$-flat.

Since it is decidable whether a pair $\langle c,d\rangle$ of rationals satisfies $a_n < c < d < b_n$ and $|d - c| < \left(\frac{1}{2}\right)^{n+1}$ and since, moreover, it is decidable whether such a pair is locator-$(n+1)$-flat, we can choicelessly choose $\langle a_{n+1},b_{n+1}\rangle$ to be the first locator-$(n+1)$-flat pair satisfying those inequalities in some fixed enumeration of the pairs of rational numbers. We then have that $\langle a_{n+1},b_{n+1} \rangle$ is $(n+1)$-flat, which supplies the induction hypothesis for the next step of the construction.  

Finally, by induction, this procedure builds the required sequences $(a_n)_{n \in \Nb}$ and $(b_n)_{n \in \Nb}$. 
\end{proof}

\begin{lem}\label[lem]{approxFlat}
  If $f:[a,b] \to \mathbb{R}$ is differentiable at $c \in (a,b)$ and there are sequences $(a_n)_{n \in \mathbb{N}}$ and $(b_n)_{n \in \mathbb{N}}$ limiting to $c$ from below and above, respectively, such that $\left|\frac{f(b_n)-f(a_n)}{b_n-a_n} \right| < \e$ for all $n$, then $|f'(c)| \leq \e$.
\end{lem}

\begin{proof}
  Since $f$ is differentiable at $c$, we have that for any $\delta>0$, there is a $\gamma>0$ such that if $|h| < \gamma$, then $|f(c+h) - f(c) - f'(c)h|< \delta |h|$. By the triangle inequality, this implies that if $|h|,|h'| < \gamma$, then 
  \[
    |f(c+h) - f(c+h') - f'(c)(h - h')| < \delta(|h|+|h'|).
  \]
  Thus, under the assumptions in the lemma, we have that $|f'(c)| \leq \e$.
\end{proof}

\begin{thm}[AMVT with Locators]\label[thm]{AMVTforLift}
  For any $a < b$ and $f: [a,b] \to \mathbb{R}$, if $f$ is pointwise continuous on $[a,b]$, pointwise differentiable on $(a,b)$ and lifts to locators, then for any $\e > 0$ there exists $c \in (a,b)$ such that 
  \[
    \left| f'(c) - \frac{f(b)-f(a)}{b-a} \right| < \e.
  \]
\end{thm}

\begin{proof}
  By replacing $f(x)$ with $f(x)-f(a)-\frac{f(b)-f(a)}{b-a}(x-a)$, we may assume that $f(a)=f(b)=0$. By extending $f$ with the value $0$, we may assume that it is actually a continuous function on all of $\mathbb{R}$.  By assumption, $f$ lifts to locators, so by Proposition \ref{approxRolle}, we can find sequences $(a_n)_{n \in \mathbb{N}}$ and $(b_n)_{n \in \mathbb{N}}$ of rational numbers such that $a_n \leq a_{n+1} < b_{n+1} \leq b_n$, $|b_n - a_n| < (\frac{1}{2})^n$, and $\left| \frac{f(b_n) - f(a_n)}{b_n - a_n} \right| < \frac{1}{2} \e$.

  Let $c$ be the element of $(a,b)$ that the sequences $(a_n)$ and $(b_n)$ both converge to. By Lemma~\ref{approxFlat}, we have that $|f'(c)| \leq \frac{1}{2}\e < \e$, as required.
\end{proof}

\begin{remark}
One thing to note is that \cite{HV} proves that the \emph{exact} Rolle's theorem (and thereby the exact mean value theorem) is provable in RCA$_0$. This relies on an application of a single non-computable instance of LEM (i.e., either $f(x)$ is nowhere locally constant or is it somewhere locally constant). This is similar to how the exact intermediate value theorem is provable in RCA$_0$ but not constructively provable. And just like with the intermediate value theorem, the non-uniform computable version can be rephrased constructively by turning the instance of LEM into an assumption. In particular, the following statement should be constructively provable: Fix $a < b$ and a pointwise continuous function $f : [a,b] \to \Rb$ such that $f(a) = f(b)$, $f$ is pointwise differentiable on $(a,b)$, and $f$ lifts to locators. If for every $a',b'$ with $a \leq a' < b' \leq b$, there are $x,y \in (a',b')$ such that $f(x) < f(y)$, then there is a $c \in (a,b)$ such that $f'(c) = 0$.
\end{remark}

\cref{AMVTforLift} has the following corollary, the analogue of which we showed to not hold in neutral constructive mathematics.  It says that the Constancy Principle holds for functions which lift to locators.

\begin{cor}[Constancy Principle with Locators]
  If $f: \bbR \to \bbR$ lifts to locators and $f'(x) = 0$ for all $x$, then $f$ is constant.
\end{cor}

\begin{proof}
  First, note that differentiability at a point implies continuity at a point, so $f(x)$ is pointwise continuous.

  Fix $x < y$ and assume for the sake of contradiction that $f(x) < f(y)$.  By the preceding theorem, we can find a $c \in (x,y)$ such that $f'(c) > 0$, contradicting the fact that $f'(c) = 0$.  By the same argument it cannot be the case that $f(x) > f(y)$. This implies that $f(x) = f(y)$.

  So, for any real $z$, we have that either $0 < z$, in which case $f(z) = f(0) = f(1)$, or $z < 1$, in which case $f(z) = f(1) = f(0)$.  So $f$ is a constant function.
\end{proof}

This shows that the law of bounded change applies to functions that are also assumed to lift to locators. 

\begin{cor}[Law of Bounded Change with Locators]
  Let $f$ be a pointwise differentiable function that lifts to locators and $|f'(x)| \leq M$ for all $x \in [a,b]$. Then $|f(b) - f(a)| \leq M|b - a|$.
\end{cor}

\begin{proof}
  By the approximate mean value theorem, we have 
  \[
    |f(b) - f(a)| \leq (M + \e) |b - a|
  \]
  for every $\e > 0$.  In particular $|f(b) - f(a)| \leq M|b - a|$.
\end{proof}

Thus we have shown that the conclusions of the AMVT, the law of bounded change, and constancy principle apply to continuous functions that lift to locators.  Since $\ACNT$ implies that all continuous functions lift to locators, we conclude the following corollary.

\begin{cor}
  $\ACNT$ (and therefore also $\ACN$ and $\ACNN$) implies the approximate mean value theorem, the law of bounded change, and the constancy principle.
\end{cor}


Many constructivists consider countable choice to be benign and therefore use it often without concern.  Many others prefer the more subdued choice principles discussed above.  Our results show that, in either case, this is enough to push the situation closer to classical mathematics than to neutral constructive mathematics on some very basic and natural questions in real analysis.

Booij describes locators as a way to compute with real numbers.  Taking this philosophy for granted, our methods say that it takes very small (but nontrivial!) amounts of countable choice to gain enough of a computational grasp of real-valued functions to be able to approximate witnesses to the Mean Value Theorem.

\section{Darboux's Theorem}

We now apply our results to prove analogues of Darboux's Theorem.  We remark that, modulo the specific version of the Mean Value Theorem invoked, what follows is essentially a standard proof of Darboux's Theorem.  Recall that it states that if $f(x)$ is differentiable, then $f'(x)$ satisfies the intermediate value theorem.  

Throughout this section, let $f$ be differentiable on an interval $[a, b]$ and let $f'$  be its derivative.  Set $c = \frac{1}{2}(a + b)$ and the functions
\[
  \alpha(t) = \begin{cases}
    a & \quad \text{if } a \leq t \leq c \\
    2t - b & \quad \text{if } c \leq t \leq b
  \end{cases}
\]
and 
\[
  \beta(t) = \begin{cases}
    2t - a & \quad \text{if } a \leq t \leq c \\
    b & \quad \text{if } c \leq t \leq b
  \end{cases}
\]
We see that $a \leq \alpha(t) \leq \beta(t) \leq b$ for all $t \in (a,b)$.  Now we define
\[
  g(t) = \frac{f(\beta(t)) - f(\alpha(t))}{\beta(t) - \alpha(t)}
\]
for $t \in (a, b)$.  Then $g$ is continuous on $(a,b)$.  Since $g(t_n) \to f'(a)$ when $t_n \to a$ and $g(t_n) \to f'(b)$ when $t_n \to b$, we can extend $g$ to a continuous function on $[a,b]$ by defining $g(a) = f'(a)$ and $g(b) = f'(b)$.  We are now positioned to state and prove the main theorems of this section.


\begin{thm}[Weak Approximate Darboux]
  If $y \in (f'(a), f'(b))$ and $\e > 0$, then there is $x_0, x_1 \in [a,b]$ with $x_0 < x_1$ such that $|\frac{f(x_1) - f(x_0)}{x_1 - x_0} - y| < \e$.  
\end{thm}

\begin{proof}
  Find, by the constructive IVT, a $t_0 \in (a,  b)$ such that $|g(t_0) - y| < \frac{\e}{2}$.  By the WAMVT, we can find $x_0, x_1 \in (\alpha(t_0), \beta(t_0))$ such that $|\frac{f(x_1) - f(x_0)}{x_1 - x_0} - g(t_0)| < \frac{\e}{2}$.  Then we have $|\frac{f(x_1) - f(x_0)}{x_1 - x_0} - y| \leq |g(t_0) - y| + |\frac{f(x_1) - f(x_0)}{x_1 - x_0} - g(t_0)| < \e$.
\end{proof}


\begin{thm}[Approximate Darboux]
  Assuming $\ACNT$, if $y \in (f'(a), f'(b))$ and $\e > 0$, then there is $x_0 \in [a,b]$ such that $|f'(x_0) - y| < \e$.  
\end{thm}

\begin{proof}
  Find, by the constructive IVT, a $t_0 \in (a,  b)$ such that $|g(t_0) - y| < \frac{\e}{2}$.  By the AMVT, we can find $x_0 \in (\alpha(t_0), \beta(t_0))$ such that $|f'(x_0) - g(t_0)| < \frac{\e}{2}$.  Then we have $|f'(x_0) - y| \leq |g(t_0) - y| + |f'(x_0) - g(t_0)| < \e$.
\end{proof}









\appendix

\section{Proof of Lemma~\ref{gConstruction}}
\label{sec:proof-of-lemma}

Given a sequence $[a_0,b_0],\dots,[a_{n-1},b_{n-1}]$ of subintervals of $[0,1]$, consider the function
\[
  g(x) = \sum_{\substack{\varnothing \neq I \subseteq \{0,\dots,n-1\} \\ I~\text{decidable}}}(-1)^{|I| + 1} \max\left\{0, \min\{x,\min_{i \in I}b_i\} -\max_{i \in I}a_i \right\}.
\]
Our goal is to show that this satisfies the requirements of Lemma~\ref{gConstruction}. 

First note that clearly $g(0) = 0$. Also note that $g(x)$ is clearly $(2^n-1)$-Lipschitz. 

Fix a natural number $M$ and find intervals $[a'_i,b'_i]$ such that
\begin{itemize}
  \item $a'_i$ and $b'_i$ are of the form $\frac{k}{M}$ for $k \leq M$ and
  \item $|a_i - a'_i| < \frac{2}{M}$ and $|b_i - b'_i| < \frac{2}{M}$ for each $i < n$.
\end{itemize}
Let $g'(x)$ be defined in the same way as $g(x)$ with $a'_i$ in place of $a_i$ and $b'_i$ in place of $b_i$. Note that for any $x \in [0,1]$,
\[
  |g(x) - g'(x)| < (2^n - 1)\cdot 2 \cdot \frac{2}{M} = \frac{4(2^n-1)}{M}.
\]

By the inclusion-exclusion principle we have for each $k \leq M$ that $g'(\frac{k}{M})$ is equal to the number of indices $\ell < k$ such that $[\frac{\ell}{M}, \frac{\ell+1}{M}] \subseteq \bigcup_{i < n}[a'_i,b'_i]$.

Let $J = \{0,\frac{1}{M},\dots,\frac{M-1}{M},1\}$. We have the following:
\begin{itemize}
  \item For any $x,y \in J$ with $x \leq y$, $g'(x) \leq g'(y)$.
  \item For any $x,y \in J$, $|g'(x)-g'(y)| \leq |x - y|$.
  \item For any $i < n$ and $x,y \in J$ with $a_i + \frac{2}{M} < x \leq y < b_i - \frac{2}{M}$, $g'(y) - g'(x) = y - x$.
  \item For any $x,y \in J$ with $x \leq y$, if every point in $[x,y]$ has distance greater than $\frac{2}{M}$ from every point in $\bigcup_{i < n}[a_i,b_i]$, then $g'(x) = g'(y) = g'(z)$ for any $z \in [x,y] \cap J$.
\end{itemize}

Since we can do this for arbitrarily large $M$ and since $g(x)$ is $(2^n-1)$-Lipschitz, the first two bullet points above imply that $g(x)$ is non-decreasing and $1$-Lipschitz. The third bullet point above likewise implies that for any $i < n$ and $x,y$ with $a_i < x \leq y < b_i$, $g(y) - g(x) = y - x$. By continuity, this implies the same only assuming $a_i \leq x \leq y \leq b_i$. Finally the fourth bullet point above implies that $g(x)$ is constant on any open interval disjoint from $\bigcup_{i < n}[a_i,b_i]$.

Now to see that the conditions listed in Lemma~\ref{gConstruction} uniquely specify $g(x)$, assume that $h(x)$ satisfies the same conditions. This implies that $h(x) - g(x)$ is constant on any open interval contained in or disjoint from $\bigcup_{i < n}[a_i,b_i]$. For any $\e > 0$, there exists a sequence of open intervals $(c_0,d_0),\dots,(c_{k-1},d_{k-1})$ such that
\begin{itemize}
  \item $\bigcup_{i < n} \{a_i,b_i\} \subseteq \bigcup_{j < k}(c_j,d_j)$,
  \item $c_0 < 0 < d_0$ and $c_{k-1} < 1 < d_{k-1}$,
  \item $\sum_{j < k} (d_j - c_j) < \e$, and
  \item $d_j < c_{j+1}$ for each $j < k-1$.
\end{itemize}
The previous statements now imply that $h(x)-g(x)$ is constant on each interval $(d_j,c_{j+1})$. Since $h(x) - g(x)$ is $2$-Lipschitz, this implies that for all $x \in [0,1]$, $|h(x) - g(x)| \leq 2 \sum_{j < k} (d_j - c_j) < 2\e$. Since we can do this for any $\e > 0$, we have that $h(x) = g(x)$ for all $x \in [0,1]$.

\bibliographystyle{plain}
\bibliography{ref}


\end{document}